\documentclass [12pt, reqno]{amsart}
\usepackage{times}
\usepackage{amssymb,latexsym,amsmath,amsfonts, eucal}
\usepackage[usenames,dvipsnames,svgnames,table]{xcolor}

\newtheorem{thm}{Theorem}[section]
      \newtheorem{lemma}[thm]{Lemma}
      \newtheorem{prop}[thm]{Proposition}

      \newtheorem{rmk}[thm]{Remark}

      \numberwithin{equation}{section}

\title [Analytic sampling theory]{Analytic Kramer sampling and quasi Lagrange-type interpolation in vector valued RKHS}
\author[ Mahapatra]{Subhankar Mahapatra}

\address{
	Department of Mathematics\\
	Indian Institute of Technology Ropar\\
	140001\\
	India}

\email{ {subhankar.19maz0001@iitrpr.ac.in, subhankarmahapatra95@gmail.com}}

\author[Sarkar]{Santanu Sarkar}

\address{
	Department of Mathematics\\
	Indian Institute of Technology Ropar\\
	140001\\
	India}

\email{ {santanu@iitrpr.ac.in, santanu87@gmail.com.}}

\begin{document}




\begin{abstract}
\noindent This paper discusses an abstract Kramer sampling theorem for functions within a reproducing kernel Hilbert space (RKHS) of vector valued holomorphic functions. Additionally, we extend the concept of quasi Lagrange-type interpolation for functions within a RKHS of vector valued entire functions. The dependence of having quasi Lagrange-type interpolation on an invariance condition under the generalized backward shift operator has also been discussed. Furthermore, the paper establishes the connection between quasi Lagrange-type interpolation, operator of multiplication by the independent variable, and de Branges spaces of vector valued entire functions.
\end{abstract}

\maketitle
\section{Introduction}
\label{Intro}
The Kramer sampling theorem, which is the generalization of the well-known sampling result due to Whittaker \cite{Whittaker}, has played a significant role for the development of sampling and interpolation theory, signal analysis, and in general function theory of mathematical analysis. Suppose $I=[a,b]\subseteq\mathbb{R}$ be any closed, bounded interval and the kernel function $K(x,\mu)$ is continuous as a function of real variable $\mu$ and belongs to $L^2(I)$ for every fixed $\mu$. Now, if there exists a sequence of sampling points $\{\mu_n\}_{n\in\mathbb{Z}}$ such that $\{K(x,\mu_n)\}_{n\in\mathbb{Z}}$ is a complete orthogonal set in $L^2(I)$, then the Kramer sampling theorem \cite{Kramer} says, if  
$$f(\mu)=\int_a^b F(x)K(x,\mu)~dx$$
for some $F\in L^2[a,b]$, then 
$$f(\mu)=\sum_{n=-\infty}^\infty f(\mu_n)F_n(\mu),$$
where the sampling functions are given by
 $$F_n(\mu)=\frac{\int_a^b K(x,\mu)\overline{K(x,\mu_n)}~dx}{\int_a^b|K(x,\mu)|^2~dx}.$$
A good amount of examples of the Kramer sampling theorem can be found in connection with the self adjoint boundary value problems. In particular,  sampling associated with Sturm-Liouville problems can be found in \cite{Annaby}, \cite{Zayed1}. The extension of this Kramer sampling theorem has been done in several ways. One interesting approach is extending this theorem to holomorphic functions associated with holomorphic kernel functions. In this direction, Kramer sampling theorem for scalar valued holomorphic functions associated with scalar valued holomorphic kernel functions has been studied in \cite{Everitt}. An abstract version of Kramer sampling theorem in the context of reproducing kernel Hilbert spaces (RKHS) has been discussed in \cite{Garcia1}, and the case when the associated sampling functions $F_n$ can be written as a quasi Lagrange-type interpolation function
$$F_n(z)=\frac{H(z)}{H(z_n)}\frac{Q(z)}{(z-z_n)Q'(z_n)},$$
where $Q(z)$ is a scalar valued entire function with only simple zeros at $z_n$ and $H(z)$ is a scalar valued entire function having no zeros, considered in \cite{Garcia2}. Also, a connection between de Branges spaces of scalar valued entire functions and having quasi Lagrange-type interpolation representation of the sampling functions has been shown in \cite{Garcia2}, \cite{Garcia3}. \\
Our main goal is to introduce an abstract version of the Kramer sampling theorem for functions in a RKHS of $\mathfrak{X}$-valued holomorphic functions, where $\mathfrak{X}$ is a complex separable Hilbert space. Also, to find a vector analog of quasi Lagrange-type interpolation functions and its correlation with the de Branges spaces of $\mathfrak{X}$-valued entire functions that we introduced in our previous work \cite{Mahapatra}. In this direction, sampling and interpolation of functions in $\mathbb{C}^n$-valued Paley-Wiener spaces \cite{Avdonin}, Lagrange-type interpolation for Hilbert space valued \cite{Zayed}, and Banach space valued \cite{Han} functions are worth noting.\\
In this quest, our motivation comes from one work of Zayed \cite{Zayed}, where he talked about the sampling theorem of $\mathfrak{X}$-valued functions. Also, it is important to note that a similar approach for scalar valued functions can be found in \cite{Zayed2}. Let $\{z_n\}_{n=1}^\infty\subseteq\mathbb{C}$ be such that $|z_n|\to\infty$ as $n\to\infty$ and $\{u_n\}_{n=1}^\infty$ be an orthonormal basis of $\mathfrak{X}$. Suppose $Q(z)$ is a scalar valued entire function having only simple zeros at $\{z_n\}_{n=1}^\infty$, then for each $z\in\mathbb{C}$ we define the following operator on $\mathfrak{X}$:
\begin{equation}
\label{Zayed op}
F(z) = \left\{
    \begin{array}{ll}
         Q(z)\frac{\langle\cdot,u_n\rangle_\mathfrak{X}}{z-z_n}u_n  & \mbox{if } z \neq z_n \\
         Q'(z_n)\langle\cdot,u_n\rangle_\mathfrak{X}u_n & \mbox{if } z = z_n.
    \end{array} \right.
\end{equation}
 Thus, for every $u\in\mathfrak{X}$, $f_u(z)=F(z)u$ is a function from $\mathbb{C}$ to $\mathfrak{X}$. Now, in this context, we mention the sampling result due to Zayed \cite{Zayed} in the following theorem:
 \begin{thm}
 Suppose $F(z)$ is the linear operator on $\mathfrak{X}$ for all $z\in\mathbb{C}$ as defined in $(\ref{Zayed op})$, then the following implications hold:
 \begin{enumerate}
 \item $F(z)$ is a bounded linear operator on $\mathfrak{X}$, for every $z\in\mathbb{C}$ and $||F(\cdot)||$ is uniformly bounded on every compact subset of $\mathbb{C}$.
 \item $f_u$ is an $\mathfrak{X}$-valued entire function for all $u\in\mathfrak{X}$ and can be recovered form its values $\{f_u(z_n)\}_{n=1}^\infty$ by the following Lagrange-type interpolation formula:
 $$f_u(z)=\sum_{n=1}^\infty \frac{Q(z)}{(z-z_n)Q'(z_n)}f_u(z_n),\hspace{.4cm} z\in\mathbb{C}.$$
 \end{enumerate}
 \end{thm}
In this article, our approach is to consider an arbitrary function $F(z)$ of bounded linear operators on $\mathfrak{X}$ and consider a RKHS $\mathcal{H}_F$ such that the elements are of the form $F(z)u, u\in\mathfrak{X}$. Then to study analogous Kramer sampling representation and quasi Lagrange-type interpolation representation for functions in $\mathcal{H}_F$.\\
Throughout this article, we consider $\mathfrak{X}$ as a complex separable Hilbert space and $B(\mathfrak{X})$ as the collection of all  bounded linear operators on $\mathfrak{X}$. Suppose $F:\Omega\subseteq \mathbb{C}\to B(\mathfrak{X})$ be any function. We recall the construction of a RKHS $\mathcal{H}_F$ of $\mathfrak{X}$-valued functions related to $F$ in section \ref{Prele}. Also, in this section, we recapture some preliminary results regarding the operator of multiplication by the independent variable $\mathfrak{T}$ and the generalized backward shift operator $R_z$ in the context of RKHS of $\mathfrak{X}$-valued entire functions. Section \ref{Kramer thm} considers $\mathcal{H}_F$ the RKHS of $\mathfrak{X}$-valued holomorphic functions on $\Omega$ and discusses an arbitrary Kramer sampling representation for functions $f\in\mathcal{H}_F$. A vector version of quasi Lagrange-type interpolation of sampling functions has been discussed in section \ref{Lagrange interpolation}. The connection between quasi Lagrange-type interpolation and the generalized backward shift operator is also shown in this section. In the last section, a connection between quasi Lagrange-type interpolation and de Branges spaces of vector valued entire functions is discussed.\\
The following notations, along with the standard symbols and the symbols already been introduced, will be used in this paper: \\
$\mathbb{C}_+$ denotes the complex upper half-plane and $\rho_\gamma(z)=-2\pi i(z-\overline{\gamma})$. For any operator $T$, $\ker T$ denotes the kernel of $T$ and $\mbox{rng}~T$ denotes the range of $T$. Also, $T^*$ denotes the adjoint operator for $T$ and $\sigma(T)$ denotes the spectrum of $T$. Any complex number $\gamma\in\mathbb{C}$ is said to be a point of regular type of $T$ if there exists $C_\gamma>0$ such that 
$$||(T-\gamma I)f||\geq C_\gamma ~||f||,$$
for all $f\in\mathcal{D}(T)$, domain of $T$. An operator $T$ is called regular if every number $\gamma\in\mathbb{C}$ is a point of regular type of $T$.\\
An operator $T\in B(\mathfrak{X})$ is said to be a Fredholm operator if the dimensions of $\ker T$ and $\ker T^*$ are both finite. The collection of all Fredholm operators on $\mathfrak{X}$ is denoted as $\Phi(\mathfrak{X})$.
\section{RKHS based on operator valued functions}
\label{Prele}
In this section, we recall the construction of a RKHS based on a $B(\mathfrak{X})$-valued function. Also, we mention some basic results about the operator of multiplication by the independent variable (multiplication operator later on) $\mathfrak{T}$ and the generalized backward shift operator $R_z$. A detailed study about the reproducing kernel Hilbert spaces can be found in \cite{Paulsen}. Assume that $F$ is any $B(\mathfrak{X})$-valued function on $\Omega\subseteq\mathbb{C}$, i.e., $F(z)\in B(\mathfrak{X})$ for all $z\in\Omega$ and $\mathcal{F}(\Omega,\mathfrak{X})$ is the collection of all functions from $\Omega$ to $\mathfrak{X}$. Now, let us define  a mapping $L:\mathfrak{X}\to \mathcal{F}(\Omega,\mathfrak{X})$ defined by $L(u)=f_u$, where
\begin{equation}
\label{1}
f_u(z)=F(z)u,\hspace{.3cm}\mbox{for all}~z\in\Omega~\mbox{and}~u\in\mathfrak{X}.
\end{equation}
It is clear that the mapping $L$ is linear and denote $\mathcal{H}_F=L(\mathfrak{X})$ ($\mathcal{H}$ when there is no confusion about the involvement of $F$). Now, we show that $\mathcal{H}_F$ can be endowed with an inner product such that it will become a RKHS. Consider
$$H:=\{u\in\mathfrak{X}:L(u)=0\}=\cap_{z\in\Omega}\ker F(z).$$
Since $H$ is a closed subspace of $\mathfrak{X}$ the quotient space $\mathfrak{X}/H$ is a Banach space corresponding to the norm
$$||\overline{u}||_{\mathfrak{X}/H}:=\inf\{||u+h||_\mathfrak{X}:h\in H\},$$
where $\overline{u}=\{u+h:h\in H\}$ is the coset of $u\in\mathfrak{X}$. Now, we define the norm in $\mathcal{H}_F$ by
$$||f_u||_{\mathcal{H}_F}:=||\overline{u}||_{\mathfrak{X}/H}=\inf\{||u+h||_\mathfrak{X}:h\in H\}=\inf\{||u||_\mathfrak{X}:f_u=L(u)\}.$$
It can be easily shown that the above infimum is indeed attained, i.e., for $L(u)=f_u\in\mathcal{H}_F$ there exists $\tilde{u}\in H^\perp$ such that 
$$||\overline{u}||_{\mathfrak{X}/H}=||\tilde{u}||_\mathfrak{X}=||f_u||_{\mathcal{H}_F}.$$

\begin{lemma}
Let $f_u,f_v\in\mathcal{H}_F$ corresponding to $u,v\in\mathfrak{X}$ such that $||f_u||_{\mathcal{H}_F}=||\tilde{u}||_\mathfrak{X}$ and $||f_v||_{\mathcal{H}_F}=||\tilde{v}||_\mathfrak{X}$. Then
\begin{enumerate}
\item $||f_u+f_v||_{\mathcal{H}_F}=||\tilde{u}+\tilde{v}||_\mathfrak{X}$ and $||f_u-f_v||_{\mathcal{H}_F}=||\tilde{u}-\tilde{v}||_\mathfrak{X}$.
\item $||f_u+if_v||_{\mathcal{H}_F}=||\tilde{u}+i\tilde{v}||_\mathfrak{X}$ and $||f_u-if_v||_{\mathcal{H}_F}=||\tilde{u}-i\tilde{v}||_\mathfrak{X}$.
\end{enumerate}
\end{lemma}
Now, by using the above lemma and the polarization identity we can define the following inner product on $\mathcal{H}_F$ by
\begin{equation}
\label{Inner product}
\langle f_u,f_v\rangle_{\mathcal{H}_F}:=\langle \tilde{u},\tilde{v}\rangle_\mathfrak{X}, \mbox{where}~||f_u||_{\mathcal{H}_F}=||\tilde{u}||_\mathfrak{X}~\mbox{and}~||f_v||_{\mathcal{H}_F}=||\tilde{v}||_\mathfrak{X}.
\end{equation}
Thus the linear map $L:H^\perp\to\mathcal{H}_F$ is a bijective isometry, i.e., a unitary operator. Hence $\mathcal{H}_F$ is a Hilbert space.
\begin{prop}
\label{2}
Let $F$ be any $B(\mathfrak{X})$-valued function on $\Omega$ and $L$ is the linear map as defined in (\ref{1}). Then the following assertions are equivalent:
\begin{enumerate}
\item $L$ is an isometry.
\item $L$ is one-one.
\item $\cap_{z\in\Omega}\ker F(z)=\{0\}$.
\item $\cup_{z\in\Omega}~\mbox{rng}~F(z)^*$ is complete in $\mathfrak{X}$.
\end{enumerate}
\end{prop}
\begin{proof}
$(1)\iff(2)$ is straight forward. Now, suppose $L$ is one-one, then $f_u=0$ implies $u=0$. Since for any $v\in \cap_{z\in\Omega}\ker F(z)$, $f_v=0$, $v$ must be zero vector. This gives $(2)\Rightarrow (3)$. Suppose $u\in\mathfrak{X}$ is such that $\langle u, F(z)^*v\rangle_\mathfrak{X}=0$ for all $z\in\Omega$ and $v\in\mathfrak{X}$. This implies $u\in \cap_{z\in\Omega}\ker F(z)$. Thus $(3)\Rightarrow (4)$. Suppose for some $u\in\mathfrak{X}$, $f_u=0$. This implies $F(z)u=0$ for all $z\in\Omega$. Thus for all $z\in\Omega$ and $v\in\mathfrak{X}$ we have $0=\langle F(z)u,v\rangle_\mathfrak{X}=\langle u,F(z)^*v\rangle_\mathfrak{X}$, consequently, $(4)\Rightarrow (2)$.
\end{proof}
In particular, if there exists a sequence $\{z_n\}_{n=1}^\infty$ in $\Omega$ such that $\cup_{n=1}^\infty \mbox{rng}~F(z_n)^*$ is complete in $\mathfrak{X}$, then also $L$ is an one-one linear map. Suppose for any $u\in\mathfrak{X}$, $L(u)=f_u$ and $||f_u||_{\mathcal{H}_F}=||\tilde{u}||_\mathfrak{X}$. Then for any $z\in\Omega$,
$$||f_u(z)||_\mathfrak{X}=||f_{\tilde{u}}(z)||_\mathfrak{X}=||F(z)\tilde{u}||_\mathfrak{X}\leq ||F(z)||~||\tilde{u}||_\mathfrak{X}=||F(z)||~||f_u||_{\mathcal{H}_F}.$$
This implies that the point evaluation linear maps are bounded in $\mathcal{H}_F$ for all $z\in\Omega$. Thus $\mathcal{H}_F$ is a RKHS of $\mathfrak{X}$-valued functions on $\Omega$. The reproducing kernel of $\mathcal{H}_F$ is denoted as $K$ and is given by $K_\gamma(z)=F(z)F(\gamma)^*$ for all $z, \gamma\in\Omega$. In fact, 
\begin{enumerate}
\item For any $u\in\mathfrak{X}$ and $\gamma\in\Omega$, $K_\gamma u\in\mathcal{H}_F$ as $L(F(\gamma)^*u)=K_\gamma u$.
\item For every $f=f_u\in\mathcal{H}_F$ with $||f||_{\mathcal{H}_F}=||\tilde{u}||_\mathfrak{X}$, $\gamma\in\Omega$ and $v\in\mathfrak{X}$,
\begin{align*}
\langle f,K_\gamma v\rangle_{\mathcal{H}_F}
&=\langle \tilde{u},F(\gamma)^*v \rangle_\mathfrak{X}\\
&=\langle F(\gamma)\tilde{u},v\rangle_\mathfrak{X}\\
&=\langle f(\gamma),v\rangle_\mathfrak{X}.
\end{align*}
\end{enumerate}
In the rest of this section, we recall some results regarding the multiplication operator $\mathfrak{T}$ and the generalized backward shift operator $R_z$ on a RKHS of vector valued entire functions. Suppose $\mathcal{H}$ is a RKHS of $\mathfrak{X}$-valued entire function, and $K_\gamma(z)$ is the corresponding $B(\mathfrak{X})$-valued RK. Then for any $\beta \in\mathbb{C}$, the set $\mathcal{H}_\beta=\{f\in\mathcal{H}:f(\beta)=0\}$ is a closed subspace of $\mathcal{H}$, and the multiplication operator $\mathfrak{T}$ is a closed operator with domain $\mathcal{D}\subseteq\mathcal{H}$.
\begin{lemma}
\label{Regular}
Suppose $\mathcal{H}$ is a nonzero RKHS of $\mathfrak{X}$-valued entire functions, and $K_\gamma(z)$ is the corresponding RK. Then for any $\beta\in\mathbb{C}$,
\begin{equation}
\label{domain condition}
R_\beta\mathcal{H}_\beta\subseteq\mathcal{H}\hspace{.2cm}\mbox{if and only if}\hspace{.2cm}R_\beta\mathcal{H}_\beta=\mathcal{D}.
\end{equation}
Moreover, if the condition $(\ref{domain condition})$ holds for some $\beta\in\mathbb{C}$, then the following implications have:
\begin{enumerate}
\item $R_\beta$ is a bounded linear operator from $\mathcal{H}_\beta$ to $\mathcal{H}$.
\item $\mbox{rng}~(\mathfrak{T}-\beta I)=\mathcal{H}_\beta$.
\item $\beta$ is a point of regular type for $\mathfrak{T}$.
\end{enumerate}
\end{lemma}
\begin{proof}
The proof of this lemma except $(3)$ follows from Lemma $7.1$ and Lemma $7.2$ in \cite{Mahapatra}. Also, the proof of $(3)$ can be done by using $(1)$ and the following observation:
$$R_\beta(\mathfrak{T}-\beta I)f=f\hspace{.4cm}\mbox{for all}~f\in\mathcal{D}.$$
\end{proof}
The following lemma, which will be used in section \ref{Lagrange interpolation}, gives a bijective map between $\mathcal{H}_{z_1}$ and $\mathcal{H}_{z_2}$ for $z_1\neq z_2$. 
\begin{lemma}
\label{bijection}
Suppose $\mathcal{H}$ is a nonzero RKHS of $\mathfrak{X}$-valued entire functions, and $K_\gamma(z)$ is the corresponding RK. If $R_{z_1}\mathcal{H}_{z_1}\subseteq\mathcal{H}$ and $R_{z_2}\mathcal{H}_{z_2}\subseteq\mathcal{H}$ for $z_1\neq z_2$ then $(\mathfrak{T}-z_2 I)R_{z_1}:\mathcal{H}_{z_1}\to\mathcal{H}_{z_2}$ is a bijective map.
\end{lemma}
\begin{proof}
The following observation proves the lemma:
$$(\mathfrak{T}-z_2I)R_{z_1}(\mathfrak{T}-z_1I)R_{z_2}=I_{H_{z_2}}~\mbox{and}~(\mathfrak{T}-z_1I)R_{z_2}(\mathfrak{T}-z_2I)R_{z_1}=I_{H_{z_1}}.$$
\end{proof}
\section{Analyticity and Kramer Sampling property in $\mathcal{H}$}
\label{Kramer thm}
This section investigates the situations when $\mathcal{H}_F$ would be a RKHS of $\mathfrak{X}$-valued analytic functions. Also, we consider a sufficient condition for which every element of $\mathcal{H}_F$ can be represented as a Kramer sampling series.
\begin{thm}
Suppose $\mathcal{H}$ is the RKHS corresponding to the $B(\mathfrak{X})$-valued function $F$ on the domain $\Omega\subseteq \mathbb{C}$. Then the elements of $\mathcal{H}$ are $\mathfrak{X}$-valued analytic functions on $\Omega$ if and only if $F$ is analytic on $\Omega$.
\end{thm}
\begin{proof}
The proof of this theorem follows from Theorem $1.2$ in chapter $V$ of \cite{Lay}.
\end{proof}
Now, suppose $\{u_n\}_{n=1}^\infty$ is an orthonormal basis of $\mathfrak{X}$. We consider a sequence of functions $\{F_n\}_{n=1}^\infty$ in $\mathcal{H}$ defined by $$F_n(z)=F(z)u_n\hspace{.3cm}\mbox{for all}~z\in\Omega.$$
The next theorem gives another criterion of $\mathcal{H}$ being a RKHS of $\mathfrak{X}$-valued analytic functions on $\Omega$ in terms of the analyticity of the sequence of functions $\{F_n\}_{n=1}^\infty$.
\begin{thm}
Suppose $\mathcal{H}$ is the RKHS corresponding to the $B(\mathfrak{X})$-valued function $F$ on the domain $\Omega\subseteq \mathbb{C}$. Then the elements of $\mathcal{H}$ are $\mathfrak{X}$-valued analytic functions on $\Omega$ if and only if the sequence of functions $\{F_n\}_{n=1}^\infty$ are analytic on $\Omega$ and $||F(.)||$ is bounded on every compact subset of $\Omega$.
\end{thm}
\begin{proof}
When $\mathcal{H}$ is a RKHS of $\mathfrak{X}$-valued analytic functions on $\Omega$, it is evident that the $F_n$'s are analytic functions and $||F(.)||$ is bounded on every compact subset of $\Omega$. Now, we prove the converse part. For any $u\in\mathfrak{X}$, $u=\sum_{n=1}^\infty \langle u,u_n\rangle_\mathfrak{X}u_n$. Then
\begin{align}
f_u(z)=F(z)u &= F(z)\sum_{n=1}^\infty \langle u,u_n\rangle_\mathfrak{X}u_n\nonumber\\
&=\sum_{n=1}^\infty \langle u,u_n\rangle_\mathfrak{X} F(z)u_n
= \sum_{n=1}^\infty \langle u,u_n\rangle_\mathfrak{X} F_n(z).\label{Analyticity}
\end{align} 
Now, for any $p\in\mathbb{N}$, we have
\begin{align*}
||\sum_{n=1}^p \langle u,u_n\rangle_\mathfrak{X} F_n(z)||_\mathfrak{X}^2&= ||F(z)\sum_{n=1}^p \langle u,u_n\rangle_\mathfrak{X}u_n||_\mathfrak{X}^2\\
&\leq ||F(z)||^2||\sum_{n=1}^p \langle u,u_n\rangle_\mathfrak{X}u_n||_\mathfrak{X}^2\\
&=||F(z)||^2\sum_{n=1}^p|\langle u,u_n\rangle_\mathfrak{X}|^2
\leq||F(z)||^2||u||_\mathfrak{X}^2.
\end{align*}
This implies the partial sums of the series in (\ref{Analyticity}) are analytic and bounded on every compact subset of $\Omega$. Hence, the elements of $\mathcal{H}$ are $\mathfrak{X}$-valued analytic functions on $\Omega$.
\end{proof}
In the remaining portion of this section, we discuss the Kramer sampling series of elements in $\mathcal{H}$. We assume that there exists a sequence $\{z_n\}_{n=1}^\infty\subseteq \Omega$ and nonzero numbers $\{c_n\}_{n=1}^\infty$ such that for all $u\in\mathfrak{X}$ the following relation holds:
\begin{equation}
F(z_n)u=c_n\langle u,u_n\rangle_\mathfrak{X}u_n\hspace{.3cm}\mbox{for all}~n\in\mathbb{N}.\label{Sampling Condition}
\end{equation}
Observe that the sequence of functions $\{F_n\}_{n=1}^\infty$ satisfies the following interpolation property at $\{z_n\}_{n=1}^\infty$:
\begin{equation}
F_n(z_m)=c_n\delta_{n,m}u_n,
\end{equation}
which is also forcing the fact that $|z_n|\to\infty$ as $n\to\infty$. Thus, if $F:\Omega\subseteq \mathbb{C}\to B(\mathfrak{X})$ satisfies (\ref{Sampling Condition}), the domain $\Omega$ should be unbounded. Also, the next identity follows from $(\ref{Sampling Condition})$ will be used frequently:
\begin{equation}
\label{Kramer}
F(z_n)^*u_n=\overline{c_n}~u_n\hspace{.3cm}\mbox{for all}~n\in\mathbb{N}.
\end{equation}
 The subsequent theorem provides a sampling series representation of elements in $\mathcal{H}$.
\begin{thm}
Suppose $\mathcal{H}$ is the RKHS corresponding to the $B(\mathfrak{X})$-valued analytic function $F$ on the domain $\Omega\subseteq \mathbb{C}$ satisfying (\ref{Sampling Condition}). Then every element $f\in\mathcal{H}$ is completely determined by the values $\{f(z_n)\}_{n=1}^\infty$ and can be reconstructed by means of the following sampling series
\begin{equation}
\label{Sampling series}
f(z)=\sum_{n=1}^\infty \langle f(z_n),u_n\rangle_\mathfrak{X}~\frac{F_n(z)}{c_n}\hspace{.3cm}\mbox{for all}~z\in\Omega.
\end{equation}
\end{thm}
\begin{proof}
Due to the relation $(\ref{Sampling Condition})$, it is clear that $L$ is an isometry. Thus the family of functions $\{F_n\}_{n=1}^\infty$ is an orthonormal basis of $\mathcal{H}$. Now, any function $f=L(u)\in\mathcal{H}$ can be written as
$$f(z)=\sum_{n=1}^\infty \langle f,F_n\rangle_\mathcal{H}F_n(z).$$ 
In addition, we have
$$\langle f,F_n\rangle_\mathcal{H}=\langle u,u_n\rangle_\mathfrak{X}=\langle u,\frac{F(z_n)^*u_n}{\overline{c_n}}\rangle_\mathfrak{X}=\frac{\langle F(z_n)u,u_n\rangle_\mathfrak{X}}{c_n}=\frac{\langle f(z_n),u_n\rangle_\mathfrak{X}}{c_n}.$$
\end{proof}

\begin{rmk}
Observe that using $(\ref{Kramer})$, the series $(\ref{Sampling series})$ can be written as the Kramer sampling series:
$$f(z)=\sum_{n=1}^\infty\langle f,K_{z_n}u_n\rangle_\mathcal{H}\frac{K_{z_n}(z)u_n}{||K_{z_n}u_n||^2},\hspace{.4cm}\mbox{for all}~f\in\mathcal{H}.$$
Thus we call the identity in $(\ref{Sampling Condition})$ as the sampling condition and the family of functions $\{F_n\}_{n=1}^\infty$ as the sampling functions.
\end{rmk}

\section{Quasi Lagrange-type interpolation property in $\mathcal{H}$}
\label{Lagrange interpolation}
In this section, we will discuss the cases when the Kramer sampling series can be written as a Quasi Lagrange-type interpolation series. Also, in this direction, we consider a special case related to symmetric operators with compact resolvent. Suppose $\mathcal{H}$ is the RKHS corresponding to the $B(\mathfrak{X})$-valued entire function $F$ satisfying $(\ref{Sampling Condition})$. Then the sampling series $(\ref{Sampling series})$ for any $f\in\mathcal{H}$ is called quasi Lagrange-type interpolation series if it has the following representation:
\begin{equation}
\label{Lagrange series}
f(z)=\sum_{n=1}^\infty\langle f(z_n),u_n\rangle_\mathfrak{X}\frac{Q(z)}{(z-z_n)Q'(z_n)}\frac{A(z)}{\langle A(z_n),u_n\rangle_\mathfrak{X}},\hspace{.5cm}z\in\mathbb{C},
\end{equation}
where $Q$ is a scalar valued entire function having only simple zeros at $\{z_n\}_{n=1}^\infty$, and $A$ is an $\mathfrak{X}$-valued entire function such that $A(z)\neq 0$ for all $z\in\mathbb{C}$. The following theorem gives a necessary and sufficient condition for the Kramer sampling series to be represented as a quasi Lagrange-type interpolation series in terms of the invariance of $\mathcal{H}_z$ under the generalized backward shift operator $R_z$ for all $z\in\mathbb{C}$.
\begin{thm}
\label{Quasi Lagrange-type series thm}
Suppose $\mathcal{H}$ is the RKHS corresponding to the $B(\mathfrak{X})$-valued entire function $F$ satisfying $(\ref{Sampling Condition})$. Then the sampling formula $(\ref{Sampling series})$ for $\mathcal{H}$ can be written as the quasi Lagrange-type interpolation series $(\ref{Lagrange series})$ if and only if $R_z\mathcal{H}_z\subseteq\mathcal{H}$ for all $z\in\mathbb{C}$.
\end{thm}
\begin{proof}
Let $R_z\mathcal{H}_z\subseteq\mathcal{H}$ for all $z\in\mathbb{C}$. We prove that $\{z_p\}_{p\neq n}$ are the only zeros of $F_n$ for every $n\in\mathbb{N}$, and these zeros of $F_n$ are all simple. Now, suppose for some $\beta\in\mathbb{C}$, $F_n(\beta)=0$, i.e., $F_n\in\mathcal{H}_\beta$, which implies $R_\beta F_n\in\mathcal{H}$. Thus $(\mathfrak{T}-z_n I)R_\beta F_n\in\mathcal{H}$ as
\begin{equation}
\label{L1}
(\mathfrak{T}-z_n I)(R_\beta F_n)(z)=\frac{z-z_n}{z-\beta}F_n(z)=F_n(z)+(\beta-z_n)(R_\beta F_n)(z),~z\in\mathbb{C}.
\end{equation}
If $\beta\neq z_p$ for all $p\in\mathbb{N}$, it is clear that $(\mathfrak{T}-z_n I)(R_\beta F_n)(z_p)=0$ for all $p\in\mathbb{N}$. Thus due to the sampling series $(\ref{Sampling series})$, we can conclude that $(\mathfrak{T}-z_n I)R_\beta F_n=0$ in $\mathcal{H}$. Since the operator $(\mathfrak{T}-z_n I)R_\beta$ is injective, we have $F_n=0$ in $\mathcal{H}$, which is a contradiction. Also, if any $z_p$, $p\neq n$ is a multiple zero of $F_n$, from $(\ref{L1})$, it is clear that $(\mathfrak{T}-z_n I)(R_\beta F_n)(z_p)=0$ for all $p\in\mathbb{N}$. Thus using the same argument as above, we can again arrive at the same contradiction situation.\\
Now, consider a scalar valued entire function $Q$ having only simple zeros at $z=z_n$ for all $n\in\mathbb{N}$. Then, due to the above reasoning, we conclude that there exists a $\mathfrak{X}$-valued entire function $A_n$ for all $n\in\mathbb{N}$ such that $A_n(z)\neq 0$ for all $z\in\mathbb{C}$ and 
$$(z-z_n)~F_n(z)=Q(z)~A_n(z),\hspace{.3cm}z\in\mathbb{C}.$$
Moreover, we find a universal $\mathfrak{X}$-valued entire function $A(z)$ such that $A(z)\neq 0$ for all $z\in\mathbb{C}$ and $A_n(z)=a_n~A(z)$ for all $z\in\mathbb{C}$, for all $n\in\mathbb{N}$ with $a_n\neq 0$. Since for $m\neq n$, the function $(\mathfrak{T}-z_nI)R_{z_m}F_n\in\mathcal{H}$ satisfies $(\mathfrak{T}-z_nI)(R_{z_m}F_n)(z)=0$ for all $z\in\{z_p\}_{p\neq m}$, by sampling series $(\ref{Sampling series})$, we have the following:
$$(\mathfrak{T}-z_nI)(R_{z_m}F_n)(z)=\frac{z-z_n}{z-z_m}F_n(z)=\langle (z_m-z_n)F_n'(z_m),u_m\rangle_\mathfrak{X}\frac{F_m(z)}{c_m},z\in\mathbb{C}.$$
Fixing $m=1$ and assuming $a_1=1$, $A(z)=A_1(z)$, we identify for every $n\geq 2$ that $A_n(z)=a_n A(z)$, where $A(z)=A_1(z)$ and $a_n=\frac{z_1-z_n}{c_1}\langle F_n'(z_1),u_1\rangle_\mathfrak{X}\neq 0$. Thus
$$
F_n(z) = \left\{
    \begin{array}{ll}
         \frac{a_nQ(z)A(z)}{z-z_n}  & \mbox{if } z \neq z_n \\
         a_nQ'(z_n)A(z_n) & \mbox{if } z = z_n.
    \end{array} \right.
$$
Also, since $F_n(z_n)=a_nQ'(z_n)A(z_n)=c_nu_n$, we have
$$c_n=a_nQ'(z_n)\langle A(z_n),u_n\rangle_\mathfrak{X}.$$
Hence, it is clear that by putting the values of $F_n(z)$ and $c_n$ in the sampling series $(\ref{Sampling series})$, one can get the required quasi Lagrange-type interpolation series $(\ref{Lagrange series})$.\\
Conversely, let the sampling formula $(\ref{Sampling series})$ for $\mathcal{H}$ can be written as a quasi Lagrange-type interpolation series $(\ref{Lagrange series})$. Suppose $f\in\mathcal{H}$ is such that $L(u)=f$ for some $u\in\mathfrak{X}$. We need to show that for any $\beta\in\mathbb{C}$ if $f\in\mathcal{H}_\beta$, i.e., $f(\beta)=0$, $R_\beta f\in\mathcal{H}$, i.e., $\frac{f(z)}{z-\beta}\in\mathcal{H}$. To be able to say that $R_\beta f\in\mathcal{H}$, it is sufficient to show that $R_\beta f$ can be written as a quasi Lagrange-type interpolation series and there exists a vector $v\in\mathfrak{X}$ such that $L(v)=R_\beta f$. The remaining proof is similar to the Theorem $3.3$ in \cite{Garcia2}. In this direction, we would only like to mention that when $\beta\not \in \{z_n\}_{n=1}^\infty$, then $\frac{f(z)}{\beta-z}=F(z)v$, where the Fourier coefficients of $v\in\mathfrak{X}$ are given by $$\langle v,u_n\rangle_\mathfrak{X}=\frac{1}{\beta-z_n}\langle u,u_n\rangle_\mathfrak{X}\hspace{.4cm}\mbox{for all}~n\in\mathbb{N}.$$
Similarly, when $\beta=z_m$ for some $m\in\mathbb{N}$, then $\frac{f(z)}{z-z_m}=F(z)w$, where the Fourier coefficients of $w\in\mathfrak{X}$ are given by
$$
\langle w,u_n\rangle_\mathfrak{X} = \left\{
    \begin{array}{ll}
         \frac{\langle u,u_n\rangle_\mathfrak{X}}{z_n-z_m}  & \mbox{if } n \neq m \\
         \frac{1}{c_m}\langle f'(z_m),u_m\rangle_\mathfrak{X} & \mbox{if } n = m.
    \end{array} \right.
$$
\end{proof}
In the rest of this section, we construct a RKHS $\mathcal{H}$ based on the resolvent operators of a symmetric operator with compact resolvent and discuss the quasi Lagrange-type interpolation series for elements in $\mathcal{H}$.
Let $T:\mathcal{D}(T)\subseteq \mathfrak{X}\to\mathfrak{X}$ is a densely defined symmetric operator such that $T^{-1}\in B(\mathfrak{X})$ and a compact operator. If $\{u_n^i\}_{i=1}^{k_n}$ are the eigenvectors of $T^{-1}$ corresponding to the eigenvalue $\xi_n$ and $z_n=\frac{1}{\xi_n}$, we recall the following basic informations:
\begin{enumerate}
\item The sequence $\{z_n\}$ is infinite and $|z_n|\to\infty$ as $n\to\infty$.
\item The orthonormal set $\{u_n^i:1\leq i\leq k_n\}_{n=1}^\infty$ is complete in $\mathfrak{X}$.
\item A number $z\in\sigma(T)$ if and only if $z\in\{z_n\}_{n=1}^\infty$ and $Tu_n^i=z_n u_n^i$.
\item For $z\not\in\sigma(T)$, the resolvent operator $R_z=(zI-T)^{-1}$ is compact and has the following form
\begin{equation}
\label{Resolvent}
R_zu=\sum_{n=1}^\infty\left[\frac{1}{z-z_n}\sum_{i=1}^{k_n}\langle u,u_n^i\rangle_\mathfrak{X}u_n^i\right]\hspace{.3cm}\mbox{for all}~u\in\mathfrak{X}.
\end{equation}
\end{enumerate}
For more details in this direction, we recommend \cite{Lay}. 
Suppose $Q(z)$ is a scalar valued entire function having only simple zeros at $z=z_n$ for all $n\in\mathbb{N}$. Then we consider the $B(\mathfrak{X})$-valued function $F(z)=Q(z)R_z$. At this point it is easy to observe that $F(z)$ is an entire function and 
\begin{equation}
\label{Point value}
 F(z_n)=Q'(z_n)\sum_{i=1}^{k_n}\langle \cdot,u_n^i\rangle_\mathfrak{X}u_n^i\hspace{.3cm}\mbox{for all}~n\in\mathbb{N}.
 \end{equation}
Thus $\cap_{n=1}^\infty\ker F(z_n)=\{0\}$, and due to Proposition \ref{2}, the operator $L$ is an isometry. We denote the corresponding RKHS $\mathcal{H}=\{F(z)u:u\in\mathfrak{X}\}$ with having reproducing kernel
$$K_\gamma(z)=Q(z)\overline{Q(\gamma)}R_zR_\gamma^*\hspace{.3cm}\mbox{for all}~\gamma,z\in\mathbb{C}.$$
Now, we want to discuss the sampling property of the elements in $\mathcal{H}$. We denote $F_n^i(z)=F(z)u_n^i$ for all $n\in\mathbb{N}$ and $1\leq i\leq k_n$. Observe that $F(z_n)^*u_n^i=\overline{Q'(z_n)}u_n^i$ holds for all $n\in\mathbb{N}$ and $1\leq i\leq k_n$. Thus every function of $\mathcal{H}$ can be recovered in terms of the sampling series like in Theorem \ref{Sampling series}. Now, to say that every function in $\mathcal{H}$ can be represented as a quasi Lagrange-type interpolation series, we only need to show that $R_z\mathcal{H}_z\subseteq\mathcal{H}$ for all $z\in\mathbb{C}$. Observe that $\mathcal{H}_{z_n}=\{0\}$ for all $n\in\mathbb{N}$. Now suppose $\beta\not\in\{z_n\}_{n=1}^\infty$ and $L(u)=f\in\mathcal{H}_\beta$ i.e., $f(\beta)=0$, which means
$$Q(\beta)\sum_{n=1}^\infty\left[\frac{1}{\beta-z_n}\sum_{i=1}^{k_n}\langle u,u_n^i\rangle_\mathfrak{X}u_n^i\right]=0.$$
Since $Q(\beta)\neq 0$, we have
\begin{align*}
f(z)&=Q(z)\sum_{n=1}^\infty\left[\frac{1}{z-z_n}\sum_{i=1}^{k_n}\langle u,u_n^i\rangle_\mathfrak{X}u_n^i\right]-Q(z)\sum_{n=1}^\infty\left[\frac{1}{\beta-z_n}\sum_{i=1}^{k_n}\langle u,u_n^i\rangle_\mathfrak{X}u_n^i\right]\\
&=(\beta-z)Q(z)\sum_{n=1}^\infty\left[\frac{1}{(z-z_n)(\beta-z_n)}\sum_{i=1}^{k_n}\langle u,u_n^i\rangle_\mathfrak{X}u_n^i\right].
\end{align*}
Thus for all $z\in\mathbb{C}$,
\begin{equation}
\frac{f(z)}{\beta-z}=Q(z)\sum_{n=1}^\infty\left[\frac{1}{(z-z_n)(\beta-z_n)}\sum_{i=1}^{k_n}\langle u,u_n^i\rangle_\mathfrak{X}u_n^i\right].
\end{equation}
Now, if we choose $v\in\mathfrak{X}$ such that the Fourier coefficients of $v$ are given by
$$\langle v,u_n^i\rangle_\mathfrak{X}=\frac{1}{\beta-z_n}\langle u,u_n^i\rangle_\mathfrak{X},\hspace{.4cm}\mbox{for all}~n\in\mathbb{N}~\mbox{and}~1\leq i\leq k_n,$$
then $\frac{f(z)}{\beta-z}=F(z)v\in\mathcal{H}$, which implies $R_z\mathcal{H}_z\subseteq \mathcal{H}$ for all $z\in\mathbb{C}$. Hence every function in $\mathcal{H}$ can be expressed as a quasi Lagrange-type interpolation series. However, in this situation, something more can be concluded. The functions of $\mathcal{H}$ can be recovered in terms of a Lagrange-type interpolation series.
\begin{thm}
Suppose $\mathcal{H}$ is a RKHS of $\mathfrak{X}$-valued entire functions corresponding to the $B(\mathfrak{X})$-valued entire function $F(z)=Q(z)R_z$. Then every element $f\in\mathcal{H}$ is completely determined by the values $\{f(z_n)\}_{n=1}^\infty$ and can be reconstructed in terms of the following Lagrange-type interpolation series
\begin{equation}
f(z)=\sum_{n=1}^\infty\frac{Q(z)}{(z-z_n)Q'(z_n)}f(z_n)\hspace{.3cm}\mbox{for all}~z\in\mathbb{C}.
\end{equation}
\end{thm} 
\begin{proof}
Suppose $f\in\mathcal{H}$ is such that $f(z)=F(z)u$ for all $z\in\mathbb{C}$ and the unique $u\in\mathfrak{X}$. Since $\{u_n^i:1\leq i\leq k_n\}_{n=1}^\infty$ is an orthonormal basis of $\mathfrak{X}$, the family $\{F_n^i:1\leq i\leq k_n\}_{n=1}^\infty$ is an orthonormal basis of $\mathcal{H}$. Then for any $z\in\mathbb{C}$, we have
\begin{equation}
\label{3}
f(z) = \sum_{n=1}^\infty\sum_{i=1}^{k_n}\langle f, F_n^i\rangle_\mathcal{H}F_n^i(z) = \sum_{n=1}^\infty\sum_{i=1}^{k_n}\langle u,u_n^i\rangle_\mathfrak{X}F_n^i(z).
\end{equation}
From $(\ref{Resolvent})$ we deduce the following
\begin{equation}
\label{4}
F_n^i(z)=F(z)u_n^i=Q(z)R_zu_n^i=\frac{Q(z)}{z-z_n}u_n^i.
\end{equation}
Now, substituting $(\ref{Point value})$ and $(\ref{4})$ on $(\ref{3})$, we get the following required Lagrange-type interpolation series:
\begin{align*}
f(z)=\sum_{n=1}^\infty\sum_{i=1}^{k_n}\langle u,u_n^i\rangle_\mathfrak{X}F_n^i(z)&=\sum_{n=1}^\infty\frac{Q(z)}{z-z_n}\sum_{i=1}^{k_n}\langle u,u_n^i\rangle_\mathfrak{X}u_n^i\\
&=\sum_{n=1}^\infty\frac{Q(z)}{(z-z_n)Q'(z_n)}f(z_n).
\end{align*}
\end{proof}
The next lemma discusses some consequences of quasi Lagrange-type interpolation in $\mathcal{H}$ related to the multiplication operator $\mathfrak{T}$. 
\begin{lemma}
Let $\mathcal{H}$ be the RKHS corresponding to the $B(\mathfrak{X})$-valued entire function $F$ satisfying $(\ref{Sampling Condition})$. Suppose every element in $\mathcal{H}$ can be written as a quasi Lagrange-type interpolation series. Then the following implications hold:
\begin{enumerate}
\item $\mathfrak{T}$ is a closed operator.
\item $\mathfrak{T}$ is a regular operator.
\item $\mathfrak{T}$ is a symmetric operator.
\item $\cap_{z:\mbox{Im}z\neq 0}\mbox{rng}(\mathfrak{T}-zI)=\{0\}$,i.e., $\mathfrak{T}$ is simple.
\end{enumerate}
\end{lemma} 
\begin{proof}
Since every element in $\mathcal{H}$ can be written as a quasi Lagrange-type interpolation series $R_z\mathcal{H}_z\subseteq\mathcal{H}$ for all $z\in\mathbb{C}$. Now, $(1)$ can be proved using the closed graph theorem, and $(2)$ follows from Lemma \ref{Regular}. Using Lemma $7.4$, the proof of $(3)$ can be realized from the proof of Theorem \ref{de Branges connection}. Now, since $\mbox{rng}(\mathfrak{T}-zI)=\mathcal{H}_z$ for all $z\in\mathbb{C}$, if $f\in\cap_{z:\mbox{Im}z\neq 0}\mbox{rng}(\mathfrak{T}-zI)$, we have $f(z)=0$ for all $z\in\mathbb{C}\setminus\mathbb{R}$. Since $f$ is an entire function $f=0$ in $\mathcal{H}$. This completes the proof.
\end{proof}

\section{Connection with the vector valued de Branges spaces}
\label{de Branges spaces}
In this section, we recall vector valued de Branges spaces that we have introduced in \cite{Mahapatra} and discuss when functions of these spaces can be represented as a quasi Lagrange-type interpolation series. Here, we tactfully choose the de Branges operators so that the corresponding de Branges spaces can be connected in this direction. The subtle change of the de Branges operators can be understood instantly, while all the other important results and notations will be unchanged. It is also worth noting that the theory of scalar valued de Branges spaces can be found in \cite{Branges 4}. Let $M=\{z_1,z_2,\ldots\}\subseteq \mathbb{R}$ be such that $|z_n|\to\infty$ as $n\to\infty$. Suppose $E_+, E_-:\mathbb{C}\to B(\mathfrak{X})$ be two entire functions such that $E_+(z), E_-(z)\in\Phi(\mathfrak{X})$ for all $z\in\Omega=\mathbb{C}\setminus M$. Also,
\begin{enumerate}
\item $E_+$ and $E_-$ both are invertible for atleast at one point in $\Omega$, and 
\item $\chi:=E_+^{-1}E_-\in\mathcal{S}^{in}\cap\mathcal{S}_*^{in}$.
\end{enumerate}
The de Branges operator is the pair of $B(\mathfrak{X})$-valued entire functions 
$$\mathfrak{E}(z)=(E_-(z),E_+(z)),z\in\mathbb{C}$$ and the corresponding positive kernel of the de Branges space $\mathcal{B}(\mathfrak{E})$ on $\mathbb{C}\times\mathbb{C}$ is given by
\begin{equation}
K_\gamma^{\mathfrak{E}}(z):= \left\{
    \begin{array}{ll}
         \frac{E_+(z)E_+(\gamma)^*-E_-(z)E_-(\gamma)^*}{\rho_\gamma(z)}  & \mbox{if } z \neq \overline{\gamma} \\
         \frac{E_+^{'} (\overline{\gamma})E_+(\gamma)^*- E_-^{'}(\overline{\gamma})E_-(\gamma)^*}{-2\pi i} & \mbox{if } z = \overline{\gamma}.
    \end{array} \right.\label{de Branges kernel}
\end{equation}
We denote the space $\mathcal{B}(\mathfrak{E})$ as $\mathcal{B}_\beta(\mathfrak{E})$ if for some $\beta\in\mathbb{C}_+$, $E_+(\beta)$ and $E_-(\overline{\beta})$ both are self adjoint. The following theorem gives a characterization of the space $\mathcal{B}_\beta(\mathfrak{E})$, and its proof is similar to the proof of Theorem $8.2$ in \cite{Mahapatra}.
\begin{thm}
Let $\mathcal{H}$ be a RKHS of $\mathfrak{X}$-valued entire functions with $B(\mathfrak{X})$-valued RK $K_\gamma(z)$ and suppose $\beta\in\mathbb{C}_+$ be such that
$$K_\beta(z),K_{\overline{\beta}}(z)\in\Phi(\mathfrak{X})\hspace{.4cm}\mbox{for all}~z\in\Omega=\mathbb{C}\setminus M$$
and
$$K_\beta(\beta),~K_{\overline{\beta}}(\overline{\beta})\hspace{.4cm}\mbox{both are invertible}.$$
Then the RKHS $\mathcal{H}$ is the same as the de Branges space $\mathcal{B}_\beta(\mathfrak{E})$ if and only if
\begin{enumerate}
\item $R_\beta\mathcal{H}_\beta\subseteq \mathcal{H}$, $R_{\overline{\beta}}\mathcal{H}_{\overline{\beta}}\subseteq \mathcal{H}$, and
\item $(\mathfrak{T}-\overline{\beta}I)R_\beta:\mathcal{H}_\beta\to\mathcal{H}_{\overline{\beta}}$ is an isometric isomorphism.
\end{enumerate}
\end{thm}
Now, if $\mathcal{H}$ is a RKHS of $\mathfrak{X}$-valued entire functions corresponding to the $B(\mathfrak{X})$-valued entire function $F$, satisfying $(\ref{Sampling Condition})$ and isometrically isometric to a de Branges space $\mathcal{B}(\mathfrak{E})$ corresponding to the de Branges operator $\mathfrak{E}(z)=(E_-(z),E_+(z))$, then $R_z\mathcal{H}_z\subseteq \mathcal{H}$ for all $z\in\mathbb{C}$ if $E_+(z)$, $E_-(z)$ both are invertible for all $z\in\mathbb{R}$. This result can be proved using techniques from Lemma $3.14$ in \cite{ArD18}, and detailed proof can be found in \cite{Mahapatra1} (Lemma $3.7$). Thus it follows from Theorem \ref{Quasi Lagrange-type series thm} that in this case, every function $f\in\mathcal{B}(\mathfrak{E})$ can be written as a quasi Lagrange-type interpolation series. The following theorem gives a converse to this result.
\begin{thm}
\label{de Branges connection}
Suppose $\mathcal{H}$ is the RKHS corresponding to the $B(\mathfrak{X})$-valued entire function $F$ satisfying $(\ref{Sampling Condition})$ and $F(z)\in\Phi(\mathfrak{X})$ for all $z\in\mathbb{C}\setminus M$. Also, there exists a $\beta\in\mathbb{C}_+$ such that $F(\beta)$ and $F(\overline{\beta})$ both are invertible. Then $\mathcal{H}$ is a de Branges space $B_\beta(\mathfrak{E})$ if the sampling series $(\ref{Sampling series})$ can be written as a quasi Lagrange-type interpolation series.
\end{thm}
\begin{proof}
We use the characterization of $\mathcal{B}_\beta(\mathfrak{E})$ to prove this theorem. Since the sampling series can be written as a quasi Lagrange-type interpolation series, $R_z\mathcal{H}_z\subseteq \mathcal{H}$ for all $z\in\mathbb{C}$. Thus in particular,
$$R_\beta\mathcal{H}_\beta\subseteq \mathcal{H}\hspace{.4cm}\mbox{and}\hspace{.4cm}R_{\overline{\beta}}\mathcal{H}_{\overline{\beta}}\subseteq \mathcal{H}.$$
Since $F(\beta)$ and $F(\overline{\beta})$ both are invertible $K_\beta(\beta)=F(\beta)F(\beta)^*$ and $K_{\overline{\beta}}(\overline{\beta})=F(\overline{\beta})F(\overline{\beta})^*$ both are invertible. Also, for any $z\in \mathbb{C}\setminus M$, $K_\beta(z)=F(z)F(\beta)^*$ and $K_{\overline{\beta}}(z)=F(z)F(\overline{\beta})^*$ both belong to $\Phi(\mathfrak{X})$. It only remains to show that $(\mathfrak{T}-\overline{\beta}I)R_\beta:\mathcal{H}_\beta\to\mathcal{H}_{\overline{\beta}}$ is an isometric isomorphism. It is clear that $(\mathfrak{T}-\overline{\beta}I)R_\beta:\mathcal{H}_\beta\to\mathcal{H}_{\overline{\beta}}$  is bijective. Now, let $f\in\mathcal{H}_\beta$, i.e., $f(\beta)=0$ is such that $f(z)=F(z)u$ for some $u\in\mathfrak{X}$. Then
$$(\mathfrak{T}-\overline{\beta}I)(R_\beta f)(z)=f(z)+(\beta-\overline{\beta})\frac{f(z)}{z-\beta},z\in\mathbb{C}.$$
We know that $R_\beta f\in\mathcal{H}$ and $(R_\beta f)(z)=F(z)v$ for all $z\in\mathbb{C}$ such that the Fourier coefficients of $v\in\mathfrak{X}$ are given by 
$$\langle v,u_n\rangle_\mathfrak{X}=\frac{1}{z_n-\beta}\langle u,u_n\rangle_\mathfrak{X},n\in\mathbb{N}.$$
Thus the following calculation completes the proof:
\begin{align*}
||(\mathfrak{T}-\overline{\beta}I)R_\beta f||_\mathcal{H}^2&=||f+(\beta-\overline{\beta})R_\beta f||_\mathcal{H}^2\\
&=||u+(\beta-\overline{\beta})v||_\mathfrak{X}^2\\
&=\sum_{n=1}^\infty |\langle u+(\beta-\overline{\beta})v,u_n\rangle_\mathfrak{X}|^2\\
&=\sum_{n=1}^\infty |\langle u,u_n\rangle_\mathfrak{X}+\frac{\beta-\overline{\beta}}{z_n-\beta}\langle u,u_n\rangle_\mathfrak{X}|^2\\
&=\sum_{n=1}^\infty |\frac{z_n-\overline{\beta}}{z_n-\beta}|^2|\langle u,u_n\rangle_\mathfrak{X}|^2\\
&=||u||_\mathfrak{X}^2=||f||_\mathcal{H}^2.
\end{align*}

\end{proof}


\noindent \textbf{Acknowledgements:} 
The research of the first author is supported by the University Grants Commission (UGC) fellowship (Ref. No. DEC18-424729), Govt. of India.
The research of the second author is supported by the MATRICS grant of SERB (MTR/2023/001324).

\noindent\textbf{Conflict of interest:}\\
The authors declare that they have no conflict of interest.


\noindent\textbf{Data availability:}\\
No data was used for the research described in this article.


\end{document}